\documentclass[12pt]{article}
\usepackage{amsmath, amsxtra, latexsym,amscd, amsthm, amssymb, indentfirst}

\usepackage{lineno}
\numberwithin{equation}{section}

\usepackage{pgf,tikz}
\usetikzlibrary{arrows}
\usetikzlibrary{matrix}
\usetikzlibrary{calc}
\usepackage{array, tabularx, longtable}
\usepackage{multicol}
\usepackage{multirow}

\newtheorem{theo}{Theorem}
\newtheorem{rem}{Remark}

\newtheorem{lem}{Lemma}
\newtheorem{definition}{Definition}
\usetikzlibrary{arrows,positioning,shapes.geometric}
\begin{document}

\title{The Dynamics of Biological models with Optimal Harvesting }

\author{Sadiq Al-Nassir\\ Department of Mathematics, College of Science\\ University of Baghdad, Iraq\\
e-mail:sadiq.n@sc.uobaghdad.edu.iq\\ }

\maketitle

\begin{abstract}
This paper aims to introduce a concept of an equilibrium point of a dynamical system which will call it almost global asymptotically stable. A biological prey-predator model is also analyzed  with a modification  function growth in prey species. The conditions of the local stable and existence of all its  equilibria are given. After that the model is extended to an optimal control problem to obtain an optimal harvesting strategy. The discrete time version of  Pontryagin's maximum principle is applied to solve the optimality problem. The characterization of the optimal harvesting variable and the adjoint variables are derived. Finally numerical simulations of various set of values of parameters are provided to confirm the theoretical findings.

\end{abstract}

\section{\small{Introduction }}
The theory of mathematical models  plays an important role for studying populations behavior. These  models can be described in continuous time case or in discrete time case by a system of ordinary differential equations or a system of deference equations respectively, this description  depends on the study problem . Discrete time systems are suitable for  populations that reproduce at specific times each month or year or each circle, this can be seen in many insects populations, marine fish, and plants. There are  variety of studies in the literature that analyzed and investigated the dynamical behavior of this kind of models, we refer to  \cite{goh,hastings,kar,hamoudi},and the references therein.

The essential concept in mathematical modeling is the stability of an equilibrium point so that an equilibrium point is called globally asymptotically stable  if the solution approaches to this equilibrium regardless of the initial condition, while it is called locally asymptotically stable if there exists a neighborhood of this equilibrium such that from every initial condition within this neighborhood the solution  approaches to it.
 Some authors in the literatures(\cite{Aziz,cushing,das,wang}) proved that the unique positive  equilibrium point in their models is global asymptotically stable even one or more than one boundary equilibrium point always exits. They ignored or excluded  the boundary points  from the domain of the function in  their biological models.In these cases they  violated the definition of the global asymptotically stable. To  treat  this situation we introduce a definition of an equilibrium point which we  call it  almost global asymptotically stable, simply means an equilibrium point is global asymptotically stable in the interior of the domain. \\ A system of difference equations may have and show  a rich and more complicated dynamical behaviors even for a simple one dimensional system. For example the logistic equation which is very known equation its  equilibria  can vary from stability behavior to  chaotic behavior\cite{may,Bischi,Puu}. Many researchers investigated and analyzed different kind of  two or more than two  dimensional models in ecology \cite{Gop,Gu,Liu}, they derived conditions for local and global stability of solutions as well as the existence of periodic solutions  \cite{Hu,Liu2010}. Some authors have discussed and assumed that the life of  populations have two stages immature and adults. However stage-structured or age structured models are are considered in the literatures  \cite{Wikan2012,wikan1995,xiao}.\\  In this study we will introduce a concept of an equilibrium point of a dynamical system as well as  biological model, prey predator models with modification growth function in prey species is investigated in details. The general form is given  by

\begin{equation}
\begin{aligned}
\label{a1}
x_{t+1}&=\frac{r x_{t }^{m}}{(1+ke^{bx_{t}})^n}-ay_{t}h(x_{t})\\
y_{t+1}& =y_{t}(-c)+dh(x_{t}))
\end{aligned}
\end{equation}
The continuous time version of model (\ref{a1}) can be written in the following form
\begin{equation}
\begin{aligned}
\label{a2}
x^{'}(t)&=\frac{r x(t)^{m}}{(1+ke^{bx(t)})^n}-ay_{t}h(x(t))\\
y^{'}(t)& =y(t)(-c)+dh(x(t)))
\end{aligned}
\end{equation}
Where $x_{t},(x(t)),\ y_{t}(y(t))$\ and $ h(x_t)$\ are the prey population density, the predator population density and the predator functional response at time $t$ respectively. The parameters $r, c,a,$and $d$  are model parameters supposing only positive values.These  parameters are the  intrinsic growth rate of prey species, the mortality  rate of predators species, the maximum per capita killing rate,and  the conversion rate  predator  respectively.While $b,m,k$\ and \ $n$  are positive constants.\\
This paper is organized as following: In section 2 we will discuss the model (\ref{a1}) with absent the predator species when $n=m=b=1$. In section 3 the prey- predator system is analyzed and all behavior of its equilibria are investigated. In section 4 the model is extended to an optimal control problem.The discrete time version of  Pontryagin's maximum principle is applied to solve the optimality problem. In section 5 numerical simulations is provided to confirm the theoretical results.  Finally  conclusion  is given.
\section{Single species model}
\begin{definition}
Consider  the following nonlinear discrete dynamical system $x_{t+1}=f(x_t),$\ where \ $f:D\subset R^n \rightarrow R^n$. An equilibrium point $x_e$, that means $f(x_e)=x_e$ is said to be almost global asymptotically stable in $D$ if it is global asymptotically stable in $D-\partial D$.\\
For the continuous time case the definition will be as follows:
Consider  the following nonlinear continuous dynamical system $x^{'}(t)=f(x(t)),$\ where \ $f: D\subset R^n \rightarrow R^n$. An equilibrium point $x_e$, that means $f(x_e)=0$ is said to be almost global asymptotically stable in $D$ if it is global asymptotically stable in $D-\partial D$.
\end{definition}
\begin{rem}
1- It is clear that every global asymptotically stable is almost global asymptotically stable, and the converse is not true. For this one can see in (\cite{das} authors proved that the positive equilibrium point $S_3=(x_1^*,x_2^*)$ there, is globally asymptotically stable only in the interior of the domain which means it is almost global asymptotically stable. However it is not global asymptotically stable because the trivial equilibrium point always exists in their model.\\
2- Clearly that  if\ $ x_e$   is  almost global asymptotically stable then it is local asymptotically stable,but  the converse is not true.For example :\\
Consider the the following system
\end{rem}
\begin{equation}
\begin{aligned}
\nonumber
x_{t+1}&=s^2x_{t }(1-x_t)(1-sx_t+sx_t^2)
\end{aligned}
\end{equation}

Where $s$ is constant parameter. If $s>3$ the system has  two positive equilibria ,namely $x^*=\frac{1+s-\sqrt{(s-1)^2-4}}{2s}\quad\text{and}\quad y^*=\frac{1+s+\sqrt{(s-1)^2-4}}{2s}$. If $s=3.1$ then $x^*=0.558$ and \ $y^*=0.7646$.\ The point  $x^*=0.558$  is locally stable which is  not almost global asymptotically stable point.\\
  Now we will investigate the dynamics of single species model of (\ref{a1}) in the absent of the predator species and $n=m=b=1$. Thus the model will be as the following
\begin{equation}
\begin{aligned}
\label{a3}
x_{t+1}&=\frac{r x_{t }}{1+ke^{x_{t}}}
\end{aligned}
\end{equation}
The model(\ref{a3}) has two equilibria , namely, the trivial equilibrium point  $x_1=0,$\ and the unique positive point $x_2=ln(\frac{r-1}{k})$. The trivial equilibrium point always exists, while the positive equilibrium  exists  when $r>k+1$. The following lemma gives the behavior of its equilibria.
\begin{lem}
For the  model (\ref{a3}) we have :\\
\begin{enumerate}
\item The trivial equilibrium point, $x_1=0,$\ is locally stable (sink) point if and only if  $r<k+1$,
and it is unstable (source) point if and only if  $r>k+1$,while it is non-hyperbolic point if and only if $r=k+1$.
\item The equilibrium point $x_2=ln(\frac{r-1}{k}),$\  is locally stable (sink) point if $k\in ((r-1)e^{-\frac{2r}{r-1}},(r-1))$. It is unstable (source) point if and only if $k< (r-1)e^{-\frac{2r}{r-1}}$,while it is non-hyperbolic point if and only if $k=(r-1)e^{-\frac{2r}{r-1}}$.
\end{enumerate}
\end{lem}
\begin{proof}
It is clear that $f^{'}(x_1)=\frac{r}{k+1}$, so that the results in 1  can be easily  obtained.
For 2 one can see $f^{'}(x_2)=\frac{r-rx_2+x_2}{r}$ then  $\mid f^{'}(x_2)\mid<1\quad \text{if and only if}\quad  (r-1)e^{-\frac{2r}{r-1}}<k<r-1$ and the results can be got.
\end{proof}
  Now we will consider a situations that population is exposition to harvest by a constant rate harvesting  which is proportional to the respective population size therefore  the model (\ref{a3}) including the harvesting will be as the following :
\begin{equation}
\begin{aligned}
\label{a44}
x_{t+1}&=\frac{r x_{t }}{1+ke^{x_{t}}}-hx_t
\end{aligned}
\end{equation}
 Where $h$ is a positive  constant representing the intensity of removing  due to  hunting or removal. It is obvious that one cannot remove  more than the population density therefore $h\leq h_{max}<1,\ h_{max}$\ is the maximum removing amount.
 The model(\ref{a44}) has also two equilibria , the trivial equilibrium point $x_o$,which always exists,and the  unique positive equilibrium $x_h=ln(\frac{r-(1+h)}{k(1+h)})$ exists only when $\frac{r-(1+h)}{k(1+h)}>1$.
 Next lemma describes the behavior of the  equilibria of model(\ref{a44}).
\begin{lem}
 For the model (\ref{a44}),the equilibria, $x_o$, and $x_h$  are
\begin{enumerate}
\item  The equilibrium point $x_0=0,$\ is locally stable(sink)  point if $r<(k+1)(1+h)$, and  it is unstable(source) if $r>(k+1)(1+h)$,while it is non-hyperbolic point if $r=(k+1)(1+h)$.
\item The equilibrium point $x_h=ln(\frac{r-(1+h)}{k(1+h)}),$\ is locally stable(sink) point if $k\in (\frac{(r-(1+h))e^{-\frac{2r}{m}}}{1+h},\frac{r-(1+h)}{1+h})$, and it is unstable(source) point if $k<\frac{(r-(1+h))e^{-\frac{2r}{m}}}{1+h}$, while it is  non-hyperbolic point if $k=\frac{(r-(1+h))e^{-\frac{2r}{m}}}{1+h}$, where $m=(1+h)(r-(1+h))$.
\end{enumerate}
\end{lem}
\begin{proof}
It is clear that $f^{'}(x)=\frac{r(1+ke^x)-rkxe^x}{(1+ke^x)^2}-h$, then $ f^{'}(x_0)=\frac{r}{(1+k)}-h$ and
$f^{'}(x_h)=\frac{r+x_h m}{r}$, therefore all results can be obtained.
\end{proof}
\section{Two Species Model,Prey-Predator Model}
In this section we will study in details the dynamics of the two species model discrete time case of model (\ref{a1}) with $n=m=b=1$. Thus the system can be written as
\begin{equation}
\begin{aligned}
\label{a5}
x_{t+1}&=\frac{r x_{t }}{(1+ke^{x_{t}})}-ay_{t}x_{t}\\
y_{t+1}& =-cy_{t}+dx_{t}y_t
\end{aligned}
\end{equation}

 The all parameter $ a,  r  , c ,d $\ and\ $ k$  are defined the  same as before. By solving the following algebraic equation one can get all equilibrium points of the model(\ref{a5}):
\begin{equation}
\begin{aligned}
\label{a6}
x&=\frac{r x}{(1+ke^{x})}-ayx\\
y& =-cy+dxy
\end{aligned}
\end{equation}
Therefore we have the following lemma.
\begin{lem}
 For all parameters values the equilibrium  points of the model (\ref{a5}) are
\begin{enumerate}
\item The trivial equilibrium point $e_0=(0,0)$\ always exists.
\item The boundary equilibrium point $e_1=(\ln(\frac{r-1}{k}),0)$\ exists only  when $r>1+k$.
\item The  unique positive equilibrium  point $e_2=(x^*,y^*)=(\frac{1+c}{d},\frac{r-(1+ke^{x^*})}{a(1+ke^{x^*})})$,which exists if $r>1+ke^{x^*}$.
\end{enumerate}
\end{lem}
In order to investigate  the dynamic behavior of the model(\ref{a5})  one has to compute the general Jacobian matrix of the model (\ref{a5}) at point (x,y). This is given by:
\begin{equation}
\label{a7}
J(x,y)=\left[\begin{array}{cc}
J_{11} &J_{12}\\
J_{21} & J_{22}
\end{array}\right]
\end{equation}
Where $J_{11}=\frac{r+rke^x-rkxe^x}{(1+ke^x)^2}-ay,\ J_{12}=-ax,\ J_{21}=dy,\ \text{and } J_{22}=dx-c$.\\

Next theorems give the local stability of $e_0, \text{and}\  e_1$ respectively.

\begin{theo}
For the model (\ref{a5}),the equilibrium point  $e_o$\ is
\begin{enumerate}
\item Locally stable (sink)  point if $r<(k+1)$, and $c<1$. It is unstable (source) point if $r>(k+1)$, and $c>1$,while it is non-hyperbolic point if $r=(k+1)\quad \text{or}\quad c=1$.\\
\item Saddle point if $r>(k+1)$, and $c<1$ or\ $r<(k+1)$, and $c>1$ .
\end{enumerate}
\end{theo}
\begin{proof}
It is clear that the Jacobian matrix at the point $e_0$ is
\begin{equation}
 \nonumber
J_{e_0}=\left[\begin{array}{cc}
\frac{r}{1+k} &0\\
0 & -c
\end{array}\right]
\end{equation}
so that the eigenvalues of $J_{e_0}$ are $\lambda_{1}=\frac{r}{1+k}$\ and $\lambda_2=-c$. Thus\ $\mid \lambda_{1}\mid<1 \ (\mid\lambda_{1}\mid>1),$\ if and only if $r<1+k\quad (r>1+k)$\ and $\mid \lambda_{2}\mid<1 \ (\mid\lambda_{2}\mid>1),$\ if and only if  $c<1\quad (c>1)$, as well as $\mid \lambda_{1}\mid=1\ \text{or}\quad\mid \lambda_{2}\mid=1\ \quad\text{if and only if  }\ r=1+k\quad\text{or}\ c=1$. Therefore the proof is finished.
\end{proof}

\begin{theo}
For model (\ref{a5}) the equilibrium point $e1$\ has the following:
\begin{enumerate}
\item The equilibrium point $e1$ is locally stable(sink)  point if $k\in I_1\bigcap I_2$.

\item The equilibrium point $e1$ is unstable (source) if $k\in I_3\bigcap I_5$
\item The equilibrium point $e1$ is saddle point if one of the following holds:\\
a)-$k\in I_1\bigcap I_5$\\
b)-$k\in I_1\bigcap I_4$\\
c)-$k\in I_3\bigcap I_2$\\
\item The equilibrium point $e1$ is non-hyperbolic point if $k=(r-1)e^{-\frac{2r}{r-1}},\quad  \text{or} \quad k=(r-1)e^{-\frac{c+1}{d}},\quad  \text{or} \quad k=(r-1)e^{-\frac{c-1}{d}} $.
\end{enumerate}
where $I_1=((r-1)e^{-\frac{2r}{r-1}},(r-1)),\quad  I_2=( (r-1)e^{-\frac{c+1}{d}},(r-1)e^{-\frac{(c-1)}{d}}),\quad I_3=(0,(r-1)e^{-\frac{2r}{r-1}})\quad I_4=((r-1)e^{-\frac{c-1}{d}},(r-1))\ \text{and}\quad I_5=(0,(r-1)e^{-\frac{c+1}{d}})$
\end{theo}
\begin{proof}
The Jacobian matrix at the point $e_1$ is
\begin{equation}
\nonumber
J_{e_1}=\left[\begin{array}{cc}
\frac{r-\ln(\frac{r-1}{k})}{r}&-a\ln(\frac{r-1}{k}) \\
d & d \ln(\frac{r-1}{k})-c
\end{array}\right]
\end{equation}
then the eigenvalues of $J_{e_{1}}$ are $\lambda_1=\frac{r-\ln(\frac{r-1}{k})}{r},\text{and}\quad \lambda_2=d \ln(\frac{r-1}{k})-c$\ then $\mid\lambda_1\mid<1\quad\text{if and only if }\quad (r-1)e^{-\frac{2r}{r-1}}<k<r-1$\ and $\mid\lambda_2\mid<1\quad\text{if and only if }\quad (r-1)e^{-\frac{c+1}{d}}<k<(r-1)e^{-\frac{(c-1)}{d}}$. Therefore all results can be obtained directly.
\end{proof}
In order to discuss the dynamic behavior of the unique positive equilibrium, the next lemma is needed.
\begin{lem}
\label{a8}
Let $F(\lambda)=\lambda^{2}+p\lambda+q$ Suppose that $F(1)>0, \lambda_{1},\lambda_{1} $, are  roots of $F(\lambda)=0$ then\\
1-$\left|\lambda_{1}\right| <1$ and  $\left|\lambda_{2}\right| <1$ if and only if $F(-1)>0$ and $q<1$\\
2-$\left|\lambda_{1}\right| >1$ and $\left|\lambda_{2}\right| <1$(or $\left|\lambda_{1}\right| <1$ and $\left|\lambda_{2}\right| >1$ )
if and only if $F(-1)<0$ \\
3-$\left|\lambda_{1}\right| >1$ and $\left|\lambda_{2}\right| >1$if and only if $F(-1)>0$ and $q>1$\\
4-$\lambda_{1}=-1$ and $\left|\lambda_{2}\right|\ne 1$ if and only if $F(-1)=0$ and $p\ne 0,2$
\end{lem}
\begin{proof}
see$\cite{xiao}$.
\end{proof}

\begin{theo}

For the unique positive equilibrium point $e_{2}$ of the model (\ref{a5}) we have :
\begin{enumerate}
\item The equilibrium point  $e_2$ is locally asymptotically stable (sink)  point if and only if  $d>N,\ \text{and}\quad r\in S$.
\item The equilibrium point  $e_2$ is unstable (source) point if and only if  $d>N,\ \text{and}\quad r>max\{\frac{M_2}{M_1},\frac{N_2}{N_1},k_1\}$.
\item The equilibrium point  $e_2$ is saddle point if and only if  $d>N,\ \text{and}\quad k_1<r<\frac{M_2}{M_1}$.
\item The equilibrium point  $e_2$ is non-hyperbolic point if these conditions are hold:\\
i)-$d\neq\ \frac{2ke^{x^*}}{k_1}$\\
ii)$r=\frac{M_2}{M_1}$\\
iii) $r\neq \frac{2 k_{1}^{2}}{kx^* e^{x^*}}\quad \text{or}\quad r\neq \frac{4k_{1}^{2}}{kx^*e^{x^*}}$.
\end{enumerate}
Where $k_1=1+ke^{x^*},\quad N=\frac{2ke^{x^*}}{k_1}\quad S=(Max\{k_1,\frac{M_2}{M_1}\},\frac{N_2}{N_1}),\quad M_1=dk_1-2x^*ke^{x^*},\ M_2=dx^* k_{1}^{2}-4k_{1}^{2},\quad N_1=d k_1-ke^{x^*},\quad \text{and}\quad \ N_2=d k_{1}^{2}$
\end{theo}
\begin{proof}
The Jacobian matrix at the unique positive equilibrium point is given by
\begin{equation}
\nonumber
J_{e_2}=\left[\begin{array}{cc}
\frac{rk_1 -rx^*ke^{x^*}}{k_{1}^{2}}-ay^*&-ax^* \\
dy^* &1
\end{array}\right]
\end{equation}
So that the characteristic polynomial of $J_{e_3}$ is $$F(\lambda)=\lambda^2+p\lambda+q$$
where $p=ay^*+\frac{rx^*ke^{x^*}}{k_{1}^{2}}-\frac{r}{k_{1}}-1$\ and \ $q=\frac{r}{k_{1}}-\frac{rx^*ke^{x^*}}{k_{1}^{2}}-ay^*+adx^*y^*$.\\
It is easy to see that $F(1)=adx^*y^* $ , hence \ $F(1)>0$.\\ Now
$F(-1)>0\quad \text{if and only if}\quad 2-2ay^*+\frac{2r}{k_1}- \frac{2rx^*ke^{x^*}}{k_{1}^{2}}+adx^*y^* >0\quad \text{if and only if}\quad 4k_1^{2}-2rx^*ke^{x^*}+dx^*rk_1-dx^*k_1^{2}>0\quad \text{if and only if}\quad r(dx^*k_1-2kx^*e^{x^*})>dx^*k_1^{2}-4k_1^{2}\quad \text{if and only if}\quad  r>\frac{M_2}{M_1} \quad \text{with}\quad d>N$.  $\\
Q<1\quad \text{if and only if}\quad \frac{r}{k_1}- \frac{rx^*ke^{x^*}}{k_{1}^{2}}- ay^*+ +adx^*y^* <1\quad \text{if and only if}\quad -rkx^*e^*+dx^*rk_1-dx^*k^2_1<0 \text{if and only if}\quad  r(dx^*k_1-kx^*e^{x^*})<dx^*k_1^{2} \quad \text{if and only if}\quad   r< \frac{N_2}{N_1}\quad \text{with}\quad  d>N$.\\
According to lemma (\ref{a8}) the proof is  finished.
\end{proof}
\section{An optimal harvesting approach}
In this section we will extend the model(\ref{a1}) to an optimal control problem  and will discuss the optimal harvesting management
of renewable resources.We assume that the population is harvested or removed with the harvesting rate $h_t$, which represents  our control variable.\\ For the single species  the model (\ref{a3}) including the harvesting effect  becomes :
\begin{equation}
\begin{aligned}
\label{c1}
x_{t+1}&=\frac{r x_{t }}{1+ke^{x_{t}}}-h_tx_t
\end{aligned}
\end{equation}
The $ x_{t },\ r, $\ and $k$\  are defined as before. In this problem the free terminal value problem is discussed and the terminal time $T$ is specified. The aim is to maximize the following  objective functional  $$ J(h)=\sum_{t=1}^{T-1} c_1h_tx_t-c_2h^2_t$$
 where\  $c_1h_tx_t$  represents the amount of money that one has to obtain, and\ $ c_2h^2_t$ is the cost of
catching and supporting the animal.  $c_1$\ and $c_2$ are positive constants. The control variable is subject to the constraint $$0\leq h_t\leq h_{Max}$$
 Now according to the discrete version of Pontryagin's maximum principle \cite{lenhart}, the Hamiltonian functional for this problem is given by
\begin{equation}
\begin{aligned}
\label{c22}
\mathcal{H}_t&=c_1h_tx_t-c_2h^2_t +\lambda_{t+1}(\frac{r x_{t }}{1+ke^{x_{t}}}-h_tx_t),\quad t=0,1,2,....T-1
\end{aligned}
\end{equation}
Where $\lambda_{t}=c_1h_t+\lambda_{t+1}(\frac{r+r ke^{x_{t}}-r x_tke^{x_{t}}}{(1+ke^{x_{t}})^2}-h_t,\quad$ is the adjoint variable or shadow price\cite{clark}. Then the characterization of the optimal control solution is
$$h^{*}_{t} =
\left\{
\begin{array}{lll}
		0 & \mbox{if }\frac{c_1 x_{t}-\lambda_{t+1}x_t}{2c_{2}} \leq 0 \\
		\frac{c_1 x_{t}-\lambda_{t+1}x_t}{2c_{2}} & \mbox{if } 0< \frac{c_1 x_{t}-\lambda_{t+1}x_t}{2c_{2}}<h_{Max}\\
		h_{Max}& \mbox{if } h_{Max}< \frac{c_1 x_{t}-\lambda_{t+1}x_t}{2c_{2}}
\end{array}
\right.$$

 In order to extend the  two species model to an optimal control problem, the model (\ref{a5})with control harvesting variable  will become as the following:
\begin{equation}
\begin{aligned}
\label{c2}
x_{t+1}&=\frac{r x_{t }}{(1+ke^{x_{t}})}-ay_{t}x_{t}-h_tx_t\\
y_{t+1}& =-cy_{t}+dx_{t}y_t
\end{aligned}
\end{equation}
Our aim in this problem is to get an optimal harvesting amount for that we  will  maximize the following cost functional $$ J(h_t)=\sum_{t=1}^{T-1} c_1h_tx_t-c_2h^2_t$$ subject to the state equations  (\ref{c2}) with control constraint $$0\leq h_t\leq h_{Max}<1$$  All  terms and parameters are  as before.So that the Hamiltonian functional will be as the following:
\begin{equation}
\begin{aligned}
\label{a66}
\mathcal{H}_t&=\sum_{t=1}^{T-1} c_1h_tx_t-c_2h^2_t +\lambda_{1,t+1}(\frac{r x_{t }}{1+ke^{x_{t}}}-h_tx_t)\\
& +\lambda_{2,t+1}(-cy_{t}+dx_{t}y_t)
\end{aligned}
\end{equation}
Where $\lambda_{1}$ and  $\lambda_{2}$ are the adjoint functions that satisfy :
\begin{equation}
\begin{aligned}
\label{a77}
\lambda_{1,t}&=c_1h_t+\lambda_{1,t+1} (\frac{r+r ke^{x_{t}}-r x_tke^{x_{t}}}{(1+ke^{x_{t}})^2}-h_t)+\lambda_{2,t+1} dy_{t}\\
\lambda_{2,t}&= \lambda_{1,t+1} (-a_{1} x_{t})+\lambda_{2,t+1} (-c+d x_{t})\\
\lambda_{1,T}&= \lambda_{2,T}=0.
\end{aligned}
\end{equation}

Furthermore the  characterization of the optimal harvesting solution $h^{*}$ satisfies:
$$h^{*}_{t} =
\left\{
\begin{array}{lll}
		0 & \mbox{if }\frac{x_{t}(c_1-\lambda_{1,t+1}) }{2c_{2}} \leq 0 \\
		\frac{x_{t}(c_1-\lambda_{1,t+1}) }{2c_{2}} & \mbox{if } 0< \frac{x_{t}(c_1-\lambda_{1,t+1}) }{2c_{2}}<h_{Max}\\
		h_{Max}& \mbox{if } h_{Max}< \frac{x_{t}(c_1-\lambda_{1,t+1}) }{2c_{2}}
\end{array}
\right.$$
 An iterative method in \cite{lenhart}\ is used to get the optimal control  with corresponding optimal state solutions  of the above optimal control problems at time $t$   by  maximizing the Hamiltonian functional at that $t$ numerically.

\section{Numerical Simulations}
In this section we will illustrate  the theoretical findings numerically for various set of  parameters.
For the local stability of the equilibrium point $e_0$ of the model(\ref{a5}), we choose this set of values $a=.1;r=0.9;k=0.01;d=1.2;c=0.01\quad \text{and}\quad (x_0,y_0)=(0.3,0.01)$\quad so that by (1) in theorem(1) the point is sink. For the equilibrium point $e_1$ this set of values $r=1.9;a=0.1;;c=0.2,d=2,k=0.6 \quad \text{and}\quad (x_0,y_0)=(0.9,0.4)$  are used. Then according to the condition (1) in  theorem (2) the point is sink. Figures 1-2 show  the  locally stability of $e_0$ ,and  $e_1$ respectively. For the  unique positive equilibrium point $e_2$ the set of values $ r=5;a =.1; k=2; c=.61; d=3;  \quad \text{and}\quad (x0,y0)=(0.53,1.9) $\ are chosen. According to (1) in theorem (3). The local stability of $e_2$  is shown in Figure 3. Other sets of values can be chosen to show the local stability of $e_0,\ e_1\quad \text{and}\quad  e_2$. \\
 For the optimal control problem an iterative numerical method in \cite{lenhart} is used to determine the optimal solutions with corresponding state solutions. For the control problem of single species we choose this set of values of parameters  $r =1.999; k=.8;c1=0.1;c2=0.01,x_0=0.1\quad \text{and}\quad T=80$ so that the total optimal objective functional $Jopt$ is found equal to $0.0412$. In Figure 4 the prey population density with control , without control and with constant harvesting is plotted. Figure 5 shows the control variable as a function of time.\\
 For the optimal control problem of two species model we choose this set of values of parameters $ a=0.1; k=2.1; c=0.5; d=2.9; r=5.2;c1=0.025;c2=0.08, (x_0,y_0)=(0.5,0.8) \quad\text{and}\quad T=80 $. So that the total optimal objective functional $Jopt$ is found equal to  $0.04121$. Figures 6-7  shows the prey population and the predator population  with control ,without control and with constant control respectively, and  Figure 8 shows the control variable as a function of time.\\ Finally Table 1 contains  the total optimal objective functional  and other different total harvesting amount strategies of both control problems by using the same values of the parameters in each problem.

\begin{figure}[ph]
\includegraphics[scale=0.4]{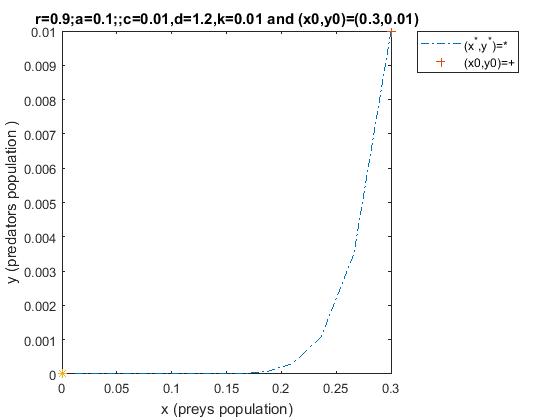}
\caption{\small{This figure shows the local stability of the equilibrium point $e_0$ of the model (\ref{a5}). }}
\vspace{-.2cm}
\end{figure}
\newpage
\begin{figure}[ph]
\includegraphics[scale=0.4]{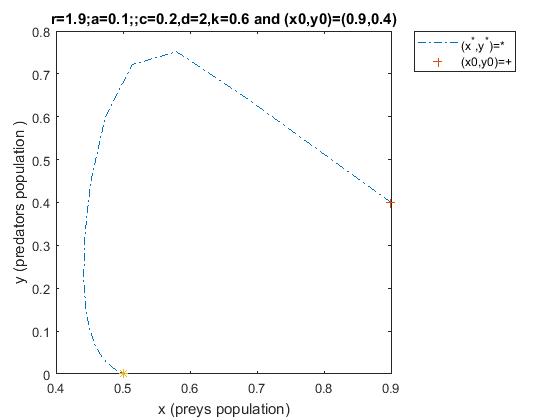}
\caption{\small{This figure shows the local stability of the equilibrium point $e_1$ of the model (\ref{a5})} }
\vspace{-.2cm}
\end{figure}
\begin{figure}[ph]
\includegraphics[scale=0.4]{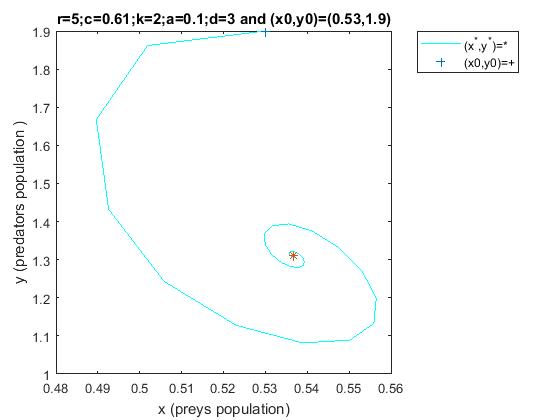}
\caption{\small{This figure shows the local stability of the unique equilibrium point $e_2$ of the model (\ref{a5})} }
\vspace{-.2cm}
\end{figure}
\newpage
\begin{figure}[ph]
\includegraphics[scale=0.4]{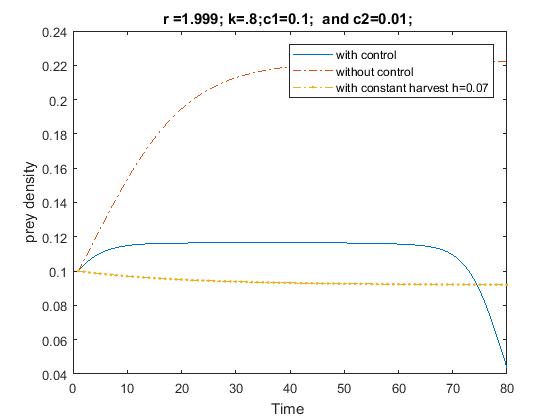}
\caption{\small{This figure shows the  affect of different  harvesting variables on the prey density in model (\ref{c1}). All values of parameters are the same in all cases.}  }
\vspace{-.2cm}
\end{figure}
\begin{figure}[ph]
\includegraphics[scale=0.4]{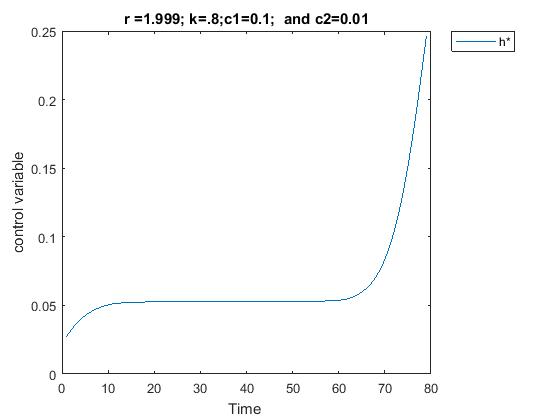}
\caption{\small{The optimal control variable of single species problem is shown function of time. }}
\vspace{-.2cm}
\end{figure}
\newpage
 \begin{figure}[ph]
\includegraphics[scale=0.4]{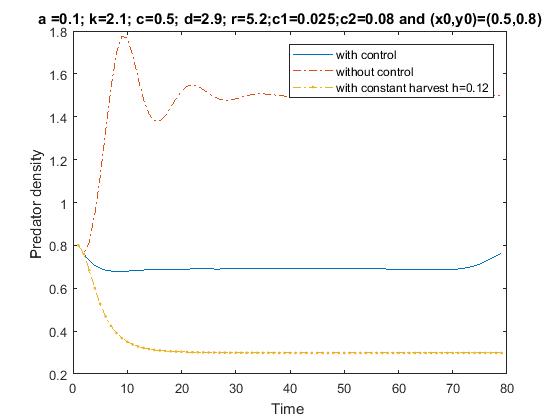}
\caption{\small{This figure shows the  affect of harvesting variable on the predator density in model (\ref{c2}). All values of parameters are the same}  }
\vspace{-.2cm}
\end{figure}
\begin{figure}[ph]
\includegraphics[scale=0.4]{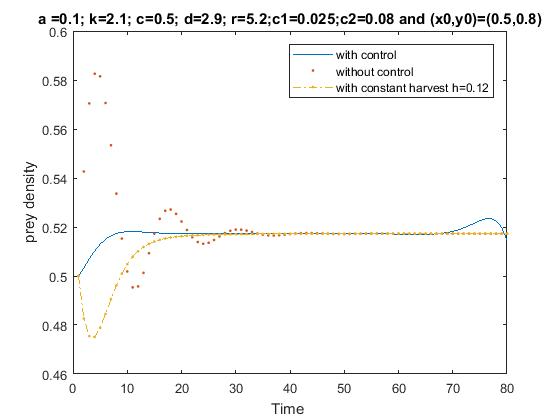}
\caption{\small{This figure shows the  affect of different harvesting variables on the prey density in model (\ref{c2}). All values of parameters are the same in all cases.}  }
\vspace{-.2cm}
\end{figure}
\begin{figure}[ph]
\includegraphics[scale=0.4]{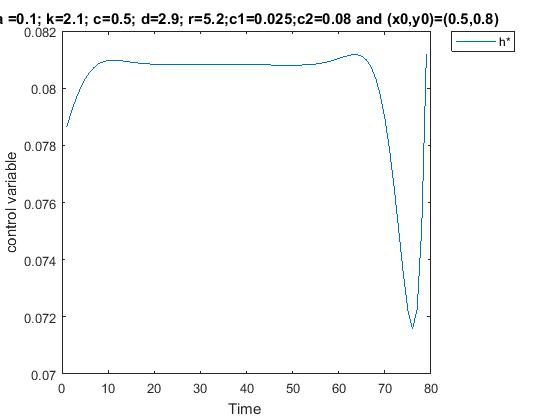}
\caption{\small{This shows the  optimal control solution of two species problem as function of time}}
\vspace{-.2cm}
\end{figure}
\newpage
    \begin{table}[ht]
\begin{center}
{\begin{tabular}{|l|l|l|l|l|l|}
    \hline
          \multicolumn{2}{c|}{\scriptsize{single species model}} &
      \multicolumn{2}{c|}{\scriptsize{Two species model}}  \\
        \cline{1-4}
   \scriptsize{ The harvesting  variable}&\scriptsize{The objective functional (J) }&\scriptsize{The harvesting  variable }&\scriptsize{The objective functional(J)}  \\    \hline

    \scriptsize{$h_t=h*$}&\scriptsize{$Jopt=0.0491$ }&\scriptsize{$h_t=h*$ }&\scriptsize{$Jopt=0.04121$}  \\    \hline
    \scriptsize{$h_t=0.065$}&\scriptsize{$J=0.0449$ }&\scriptsize{$h_t=0.12$ }&\scriptsize{$J=0.0306$}  \\    \hline
   \scriptsize{$h_t=0.06$}&\scriptsize{$J=0.04524$ }&\scriptsize{$h_t=0.1$ }&\scriptsize{$J=0.0386$}  \\    \hline
   \scriptsize{$h_t=0.058$}&\scriptsize{$J=0.04522$ }&\scriptsize{$h_t=0.09$ }&\scriptsize{$J= 0.04052$}  \\    \hline
   \scriptsize{$h_t=0.055$}&\scriptsize{$J=0.0450$ }&\scriptsize{$h_t=0.08$ }&\scriptsize{$J= 0.04118$}  \\    \hline
   \scriptsize{$h_t=0.05$}&\scriptsize{$J=0.0442$ }&\scriptsize{$h_t=0.07$ }&\scriptsize{$J= 0.04054$}  \\    \hline
     \end{tabular}}
      \caption{\small{This table shows the results of the optimal harvesting amount with different  constant harvesting. }}
      \end{center}
\end{table}
\newpage
\section{Conclusion}
In this paper the definition of almost global asymptotically stable of an equilibrium point is introduced with examples. We have also investigate biological models in discrete time case for one species  with a modification growth function and  for two species, prey-predator, model.We have studied the dynamics behavior of the  equilibrium points for each model. after that we have extended these model to an optimal control problems. We have used the Pontaygin's principle maximum to get the optimal solutions numerically. Finally one can investigate the dynamics behavior of the continuous time case as well as  the other values of $n,m,$\ and $b$ in future work.

\bibliographystyle{unsrt}
\bibliography{sadiqalnassir}
\end{document}